\documentclass[11pt,reqno,a4paper]{amsart}
\usepackage{amsfonts,hyperref}
\hypersetup{
    colorlinks=true,
    linkcolor=blue,
    citecolor=red,
    pdfstartpage=,
    pdftitle={relative convexity},
}
\usepackage[dvipsnames]{xcolor}
\usepackage{mathpazo}
\usepackage{graphicx}
\usepackage{lineno}
\usepackage[margin=3cm]{geometry}

\usepackage{soul}

\newcommand{\R}{\mathbb{R}}

\newcommand{\N}{\mathbb{N}}

\newcommand{\vp}{\varphi}

\newtheorem{theorem}{Theorem}[section]

\theoremstyle{plain}

\newtheorem{corollary}[theorem]{Corollary}

\newtheorem{definition}{Definition}

\newtheorem{proposition}{Proposition}[section]

\numberwithin{equation}{section}

{\theoremstyle{definition}
	\newtheorem{remark}{Remark}
	\newtheorem{example}{Example}

}

\begin{document}

\title{On Relative Convex Sequences}

\author{A.~El Farissi\textsuperscript{1}, Z.~Latreuch\textsuperscript{1}, S.~Taf\textsuperscript{2} and M.~A.~Zemirni\textsuperscript{1}}

\address{\textsuperscript{1}~~National Higher School of Mathematics, Scientific and Technology Hub of Sidi Abdellah, P.O.Box 75, Algiers 16093, Algeria.}
\email{\href{abdallah.elfarissi@nhsm.edu.dz}{abdallah.elfarissi@nhsm.edu.dz}, \href{z.latreuch@nhsm.edu.dz}{z.latreuch@nhsm.edu.dz},  \href{amine.zemirni@nhsm.edu.dz}{amine.zemirni@nhsm.edu.dz}}

\address{\textsuperscript{2}~~Department of Mathematics, University of Mostaganem, B. P. 227, Mostaganem 27000, Algeria}
\email{\href{sabrina.taf@univ-mosta.dz}{sabrina.taf@univ-mosta.dz}}

	%
	%
\makeatletter
\@namedef{subjclassname@2020}{\textup{2020} Mathematics Subject Classification}
\makeatother

	
	\maketitle
	\begin{abstract}
	In this paper, we introduce the concept of relative convex sequences and establish their fundamental properties, highlighting their similarities to those of convex sequences. Additionally, we prove new inequalities of the Lupas and Hermite-Hadamard-Fejér type for relative convex sequences. In certain cases, and as an application, we show how the concept of relative convexity can facilitate the derivation of new inequalities for convex sequences.
\end{abstract}

\section{Introduction and Motivation}

A real-valued function $ f $ defined on an interval $ I\subset \R $ is said to be convex if, for every $ x,y \in I$ and for every $ \lambda \in [0,1] $, $ f $ satisfies
	\begin{equation*}\label{cvf}
		f(\lambda x + (1-\lambda) y) \le \lambda f(x) + (1-\lambda) f(y). 
	\end{equation*} 
Furthermore, if $ f $ is differentiable, then $ f $ is convex if and only if its derivative $ f' $ is a non-decreasing function.
A sequence $ (a_i)_{i\ge1} $ of real numbers is said to be convex if it satisfies, for every $ i \ge 2 $,
\begin{equation}\label{cv}
a_i \leq \frac{a_{i-1}+a_{i+1}}{2}.
\end{equation}
This is equivalent to saying that the sequence $(\Delta a_i)_{i \ge 1}$, where $\Delta a_i =a_{i+1} -a_{i}$, is a non-decreasing sequence. 
The concept of convex sequences is, in fact, derived from the concept of convex functions.
Indeed, if $f$ is a convex function defined on $[1,\infty)$, then the sequence $(f(i))_{i\ge1}$ is a convex sequence. Conversely, if a sequence $(a_i)_{i\ge 1}$ is convex, then the function $f$ whose graph is the polygonal line with corner points $(i,a_i)$ is a convex function on $[1,\infty)$ as noted in \cite[Remark 1.12]{PPT}.

\medskip
Another convexity phenomenon concerning real functions is the \emph{relative convexity}, where the convexity is studied with respect to other functions. In fact, a function $ f $ is said to be \emph{convex with respect to} (or briefly, w.r.t) an increasing function $ g $ if $ f\circ g^{-1} $ is a convex function. In addition, if $ f $ and $ g $ are both differentiable, then $ f $ is convex w.r.t $ g $ if and only if $ f'/g' $ is non-decreasing. For more details on relative convex functions, we refer to \cite{cargo, HLP, MPF}.  
%
%
%
As the convexity of functions led to the convexity of sequences, we naturally anticipated a similar connection when it comes to relative convexity. However, to our surprise, we could not find anything in the literature regarding relative {convexity}  for real sequences. This, in fact, motivated us to shed the light on the  phenomenon of relative convexity for real sequences.

 
%


\medskip


\begin{definition}
	Let $ (a_i)_{i\ge 1} $ be a real sequence.  We say that $ (a_i)_{i\ge 1} $ is a relative convex sequence if there exists an increasing sequence $ (b_i)_{i\ge 1} $ such that $ (a_i)_{i\ge 1} $ is convex {with respect to}  $ (b_i)_{i\ge 1} $, or equivalently, if
	\begin{equation*}\label{sab}
		\left( \frac{\Delta a_i}{ \Delta b_i} \right)_{i\ge 1} \text{  is non-decreasing.}
	\end{equation*}
	 We denote by  $ c_r $  the class of all relative convex sequences.
\end{definition}

\medskip
It is clear that the convexity implies the relative convexity,  as every convex sequence is convex w.r.t $ (i)_{i\ge 1} $. However, the inverse is not true in general, as one may see that the sequence $ (\ln i)_{i\ge 3} $ is not convex while it is convex w.r.t $ (\ln (\ln i ))_{i\ge 3} $, and this last sequence is totally increasing, see Example~\ref{ex1}. In particular, one can easily check that a sequence in $ c_r $ is convex if and only if it is convex w.r.t an arithmetic sequence with a positive common difference.  




\medskip
Now, for a sequence $a= (a_i)_{i\ge1} $ we define the class $ T_a $ of increasing sequences $ (t_i)_{i\ge1} $  for which   $ (a_i)_{i\ge1} $ is convex w.r.t  $ (t_i)_{i\ge1}$. Notice that if $ (a_i)_{i=1}^n $ is a finite relative convex sequence of $ n $ terms, then the sequences in $ T_a $ are all finite sequences of $ n $ terms.  These new classes, $ T_a $, seem important as they have several properties to investigate. For example, one may ask under what conditions on $ a,b\in c_r $ we have $ T_a \cap T_b \neq \emptyset $ or $ T_a\subset T_b$, ...? As a partial answer, if $ a $ and $ b $ are increasing and $a\in T_b$, then $T_a\subset T_b$. Indeed, this can be shown directly from the following identity
	\begin{equation*}
		\frac{\Delta b_i}{\Delta s_i} = \frac{\Delta b_i}{\Delta a_i} \cdot\frac{\Delta a_i}{\Delta s_i}, \quad (s_i)_{i\ge 1} \in T_a.
	\end{equation*}
An immediate consequence of this property is: \emph{An increasing sequence that is convex w.r.t a convex sequence is convex.}  

The rest of this paper is organized as follows.  In Section~2, we provide tow main characterizations for relative convex sequences, which fail to be true for convex sequences.  Some similar properties for convex sequences are given in Section~3. Finally, Section~4 is devoted to presenting new inequalities for relative convex sequences, where some of them are counterpart of some classical inequalities for convex sequences.

%

\section{Two main characterizations for relative convex sequences}
In this section,  we give sufficient and necessary conditions for a real sequence to be relative convex.

\subsection{Relative convex sequences and convex functions}

Let $ f $ be a convex function on $ [1,+\infty) $. We know if we choose the sequence $ t_i =i $, $ i\ge1 $, in the interval $ [1,+\infty) $, then $ (f(t_i))_{i\ge1} $ is convex.  Based on this fact, a natural question arises: What can be said on the sequence $ (f(t_i))_{i\ge 1} $ when $ f $ is convex on an interval $ I\subset \R $, and $ (t_i)_{i\ge 1} $ is an arbitrary increasing sequence in $ I $? 

Conversely, we also know that if a sequence $ (a_i)_{i\ge 1} $ is convex, then the function whose graph is polygonal line with corner points at $ (i,a_i) $ is a convex function on $ [1, +\infty) $. Now, if we assume that $ (a_i)_{i\ge 1} $ is convex w.r.t. $ (t_i)_{i\ge1} $, then what can be said on the function whose graph is polygonal line with corner points at $ (t_i,a_i) $?

The answer to these questions are given in the following result.


\begin{theorem}\label{rc-cf}
	Let $ (a_i)_{i\ge 1} $  and $ (t_i)_{i\ge1} $ be real sequences, where $ (t_i)_{i\ge1} $ is increasing. Then, $ (t_i)_{i\ge1} \in T_a$ if and only if there exists a convex function $ \vp :D_{\vp}\to\R$ such that $ (t_i)_{i\ge1} \subset D_\vp$ and $ \vp(t_i) = a_i $ for every $ i \in \N$.
\end{theorem}
\begin{proof}
	First, suppose that $ (t_i)_{i\ge1} \in T_a$. Define the function $ {a}^*_t $ to be the polygonal function with corner points at $ (t_i, a_i) $, $ i\ge 1 $. Clearly $ {a}_t^* $ is a piecewise linear function; it is defined as follows
		\begin{equation}\label{defp}
			a_t^*(x) = a_i  + \frac{\Delta a_i}{\Delta t_i} (x -t_i), \quad x\in \left[ t_i, t_{i+1}\right) , \, i=1,2, \ldots.
		\end{equation}
	We prove next that $ a_t^* $ is convex on $ \left[ t_1,+\infty\right)  $.   Let $ x,y \in  \left[ t_1,+\infty\right)  $ with $ x<y $. Then there exist positive integers $ i,j,k $ satisfying $ 1 \le i\le k\le j $ such that
	\begin{equation*}
		x\in \left[ t_i,t_{i+1}\right) , \quad y\in\left[ t_j, t_{j+1}\right) ,\quad (x+y)/2\in \left[ t_k,t_{k+1}\right) .
	\end{equation*}
Therefore, by using \eqref{defp} and by using the fact that $ (a_i)_{i\ge 1} $ is convex w.r.t $ (t_i)_{t\ge1} $ we obtain
\begin{align*}
	a_t^*(x) &=a_i+\frac{\Delta a_i}{\Delta t_i}\left(x-t_i\right)\\
	&=a_{i+1}+\frac{\Delta a_i}{\Delta t_i}\left(x-t_{i+1}\right) \\
	& \geq a_{i+1}+\frac{\Delta a_{i+1}}{\Delta t_{i+1}}\left(x-t_{i+1}\right)\ge \cdots \geq a_k+\frac{\Delta a_k}{\Delta t_k}\left(x-t_k\right), 
\end{align*}
Similarly we obtain
	\begin{align*}
		a_t^*(y) &=a_j+\frac{\Delta a_j}{\Delta t_j}\left(y-t_j\right) \\
		& \geq a_j+\frac{\Delta a_{j-1}}{\Delta t_{j-1}}\left(y-t_j\right) \\
		&=a_{j-1}+\frac{\Delta a_{j-1}}{\Delta t_{j-1}}\left(y-t_{j-1}\right) \ge \cdots\ge a_k+\frac{\Delta a_k}{\Delta t_k}\left(y-t_k\right).
	\end{align*}
	Hence
\begin{equation*}
	\frac{a_t^*(x)+a_t^*(y)}{2} \ge \frac{2 a_k+\frac{\Delta a_k}{\Delta t_k}\left(x+y-2 t_k\right)}{2}=a_k+\frac{\Delta a_k}{\Delta t_k}\left(\frac{x+y}{2}-t_k\right)=a_t^*\left(\frac{x+y}{2}\right).
\end{equation*}
{This, together with the continuity of $ a_t^* $ on $ \left[ t_1,+\infty\right)$, implies the convexity of $ a_t^* $.}  Thus, by taking $ \vp \equiv a_t^* $, we see that $ \vp $ is convex and $ \vp(t_i) =a_i $ for every $ i\ge1 $.

\medskip
Now, suppose that there exists a convex function $ \vp :D_{\vp}\to\R$ such that $ (t_i)_{i\ge1} \subset D_\vp$ and $ \vp(t_i) = a_i $  for every $ i\ge1 $. We prove that $ (t_i)_{i\ge1} \in T_a$,  i.e. $ (a_i)_{i\ge 1} $ is convex w.r.t $ (t_i)_{i\ge1} $. Let $ i\ge 1 $ be an  arbitrary positive integer. We have from the convexity of $ \vp $ (Chordal Slope Lemma),
	\begin{equation*}
		\frac{\vp(t_{i+1})-\vp(t_i)}{t_{i+1} - t_i} \le \frac{\vp(t_{i+2}) - \vp(t_{i+1})}{t_{i+2}-t_{i+1}}.
	\end{equation*}
Since $ \vp(t_i) = a_i $ for any $ i\ge 1 $, it follows that
	\begin{equation*}
		\frac{\Delta a_i}{\Delta t_i} \le \frac{\Delta a_{i+1}}{\Delta t_{i+1}}, \quad i=1,2, \ldots,
	\end{equation*}
which means that $ (\Delta a_i / \Delta t_i)_{i\ge1} $ is a non-decreasing sequence, and hence $ (t_i)_{i\ge1}  \in T_a$.
 \end{proof}

\begin{remark}
	Theorem~\ref{rc-cf} applies also to the case when $ (a_i)_{i\ge1} $ is a finite sequence. 
	  
\end{remark}

\begin{example}\label{ex1}
	The sequence $(a_i)_{i\geq 3}$ defined by $a_i= \ln i$, is convex w.r.t. $(\ln(\ln i))_{i\geq 3}$. Indeed, one may see that the convex function $\vp:\left[0,\infty \right) \to \R $ given by $\vp(t)=e^t$ satisfies 
	$
	\varphi(\ln(\ln i))=a_i,
	$
	for all $i\in \N$.
\end{example}

\subsection{Relative convexity and V-shaped sequences}
A real sequence (resp. a real function) is called V-shaped if it is either  monotonic or {non-increasing} and then {non-decreasing}. It is known that every convex sequence (resp. convex function) is V-shaped  \cite[Lemma 1]{RS} \cite[Lemma 1.1.4]{NP}. {Upon careful examination of the convexity phenomenon, it appears that a convex sequence (resp. convex function) adheres to one of the following cases}:
%
\begin{enumerate}
	\item strictly monotonic. 
	\item decreasing and then constant.
	\item constant and then increasing.
	\item decreasing and then increasing. 
	\item decreasing and then constant and then increasing. 
\end{enumerate}
In this paper,  we say that a sequence (resp. a function) is \emph{strictly V-shaped} if it satisfies one of the five cases above.
Notice that a strictly V-shaped sequence is not necessarily convex. For instance, the sequence $ (\ln i)_{i\ge1} $ is strictly monotonic, but it is not convex.  
On the other hand, we find that a sequence is relative convex if and only if it is strictly V-shaped as proved in the following result. 

\begin{theorem}\label{svs}
	A real sequence is relative convex if and only if it is strictly V-shaped.
\end{theorem}
\begin{proof}
	Let $ (a_i)_{i\ge1} $ be a real sequence.
	Suppose that $ (a_i)_{i\ge1}\in c_r $, and let $ (t_i)_{i\ge1} \in T_a $. Then the polygonal function $ a_t^* $, defined in \eqref{defp}, is convex and $ a_t^*(t_i) = a_i $ for every $ i\ge1 $. Since every convex function is strictly V-shaped, it follows that $ (a_i)_{i\ge1} $ is also strictly V-shaped.
	
	\medskip
\noindent	Suppose now that $ (a_i)_{i\ge1} $ is strictly V-shaped.  Assume that $ (a_i)_{i\ge1} $ is increasing. By letting  $ (s_i)_{i\ge1} $ be any positive increasing sequence, we define the sequence $ (t_i)_{i\ge1} $ recursively by: $ t_1 =\alpha$, where $ \alpha \in \R $ is arbitrarily chosen, and 
\begin{equation}\label{t}
	t_{i+1} = \frac{a_{i+1} - a_i}{s_i} +t_{i}, \quad i=1,2, \ldots
\end{equation} 
From this construction, it is clear that $ (t_i)_{i\ge1} $  is increasing  and $ (a_i)_{i\ge1} $ is convex w.r.t. $ (t_i)_{i\ge1} $, that is, $ (a_i)_{i\ge1}$ is relative convex. Notice that if $ (a_i)_{i\ge1} $ is finite with $ n $ terms, then $ (s_i)_{i\ge1} $ is  taken to be finite with $ n-1 $ terms.
If $ (a_i)_{i\ge1} $ is decreasing, then we take $ (s_i)_{i\ge1} $ to be any negative increasing sequence, and similarly, we define $ (t_i)_{i\ge1} $ as in \eqref{t} to get that $ (t_i)_{i\ge1}  \in T_a$.
If there exists $ m \ge1 $ such that 
	\begin{equation*}
		a_1>a_2> \cdots >a_{m-1}> a_m \quad\text{and}\quad a_m < a_{m+1} < \cdots,
	\end{equation*}  
then we take $ (s_i)_{i\ge1} $ to be any increasing sequence whose first $ m-1 $ terms are negative,~i.e., 
\begin{equation*}
	s_1 < s_2< \cdots<s_{m-1} < 0 < s_m < s_{m+1}<\cdots.
\end{equation*}
Again, we define $ (t_i)_{i\ge1} $ as in \eqref{t} to get that $ (t_i)_{i\ge1}  \in T_a$. Assume now that there exist $ m,\ell \ge 1 $ such that 
\begin{equation*}
	a_1>a_2> \cdots >a_{m-1}> a_m =\cdots= a_{m+\ell} < a_{m+\ell+1} < \cdots,
\end{equation*}  
From what we have done yet, there exists an increasing sequence $ (t_i)_{i=1}^m $ such that $  (a_i)_{i=1}^m $ is convex w.r.t  $ (t_i)_{i=1}^m $. Next, we take any numbers $ t_{m+1}, \ldots, t_{m+\ell} $ satisfying $t_m< t_{m+1} < \cdots< t_{m+\ell} $.  Since $ a_m = \cdots= a_{m+\ell} $, it remains true that  $   (a_i)_{i=1}^{m+\ell}  $ is convex w.r.t  $ (t_i)_{i=1}^{m+\ell} $. Moreover, we can construct an increasing sequence $ (t_i)_{i\ge m+\ell+1} $ for which  $ (a_i)_{i\ge m+\ell+1} $ is convex w.r.t $ (t_i)_{i\ge m+\ell+1} $ and $ t_{ m+\ell+1} > t_{m+\ell} $. Finally, we see that $  (a_i)_{i\ge1} $ is convex w.r.t $ (t_i)_{i\ge 1}$.
%
%
%
%
%
%
\end{proof}

{The following result presents an additional property of strictly V-shaped sequences and their relative convexity. }
\begin{theorem}
	Let $ (a_i)_{i=1}^n $ be a strictly V-shaped sequence. Then for any interval $ [\alpha,\beta]$ there exists a subdivision $ (t_i)_{i=1}^n $ of $ [\alpha,\beta] $ for which $ (t_i)_{i=1}^n \in T_a$.
\end{theorem}

\begin{proof}
	We consider first the case when $ (a_i)_{i=1}^n $ is increasing. Choose $ t_1 = \alpha $, $ t_n = \beta $ and~$ s_1 $ to be any number satisfying 
	\begin{equation*}
		0< s_1< \frac{a_n - a_1}{\beta - \alpha}.
	\end{equation*} 
	For   given $ t_i $ and $ s_{i-1} $ (we set $ s_0 =0 $), we define $ t_{i+1} $ as follows
	\begin{equation*}
		t_{i+1} = \frac{a_{i+1} - a_i}{s_i} + t_i, \quad i=1, \ldots, n-1,
	\end{equation*}
	where $ s_i $ is any number satisfying 
	\begin{equation*}
		s_{i-1}< s_i< \frac{a_n - a_i}{\beta - t_i}, \quad \text{for}\; i=1,\ldots,n-2,
	\end{equation*} 
	and
	$$
	s_{n-1}=\frac{a_n-a_{n-1}}{\beta-t_{n-1}}.
	$$
	One can easily check that $ (t_i)_{i=1}^n $ is increasing and it is a subdivision of $ [\alpha,\beta] $. In addition, $ (t_i)_{i=1}^n \in T_a$. 
%
	If  $ (a_i)_{i=1}^n $ is decreasing, then we follow the same construction as above. The only difference is that the sequence  $ (s_i)_{i=1}^{n-1} $ will be negative.
	If  $ (a_i)_{i=1}^n $ has one minimal term, i.e., there exists $ m \in \{2, \ldots, n-1\}$  for which  
	\begin{equation*}
		a_1> \cdots> a_m \quad\text{and} \quad a_m < \cdots <a_n,
	\end{equation*}
	then choose $ \gamma \in ]\alpha,\beta[$,  and follow the same construction as above for the sequence $ (a_i)_{i=1}^m  $ on $ [\alpha,\gamma] $ and for the sequence $ (a_i)_{i=m}^n $ on $ [\gamma,\beta] $. 
	
	\medskip
	\noindent The other cases can be handled analogously as in the proof of Theorem~\ref{svs}.
\end{proof}

We close this section with the following corollary.
\begin{corollary}
	The class $ c_r $ is not closed under the operation of addition.
\end{corollary}
This can be seen by taking the sequences: $ \sqrt{|n-3|} $ and $ \sqrt{|n-9|} $ for $ n\ge1 $. Both of them are strictly V-shaped, so by Theorem~\ref{svs}, they are relative convex. However, their sum is not strictly V-shaped, and then it is not relative convex.

\section{Analogous properties to convex sequences}

In this section, we give some results on relative convex sequences that are analogous to those of convex sequences.  

By making use of Theorem~\ref{rc-cf}, we give the following essential property of relative convex sequences.
\begin{proposition}\label{P1}
Let $(a_i)_{i\geq 1}$ be a real sequence. Then, $(t_i)_{i\geq 1}\in T_a$ if and only if the inequality 
\begin{equation}\label{grv}
a_i \leq \frac{\Delta t_{i}  }{\Delta t_{i} +\Delta t_{i-1}} a_{i-1}+\frac{ \Delta t_{i-1}  }{\Delta t_{i} +\Delta t_{i-1}}a_{i+1}
\end{equation}
holds for every $i\geq 2$. In particular, if $(a_i)_{i\geq 1}$ is increasing, then $(t_i)_{i\geq 1}\in T_a$ if and only if the inequality 
\begin{eqnarray}\label{grv2}
\frac{\Delta ^2 a_i}{\Delta  a_i} \geq \frac{\Delta^2 t_i}{\Delta t_i}
\end{eqnarray}
holds for every $i\ge1$.
\end{proposition}
\begin{proof}
To show \eqref{grv}, one has just to take the piecewise continuous function $a_t^*$ defined in \eqref{defp}, which is convex according to Theorem~\ref{rc-cf}. Hence, for any integer $i \geq 2$, we have
\begin{eqnarray}
a_i= a_t^*(t_i)&=& a_t^* (\lambda_{i} t_{i+1}+(1-\lambda_{i}) t_{i-1}) \nonumber \\
&\leq & \lambda_{i} a_t^*(t_{i+1}) +(1-\lambda_{i}) a_t^*(t_{i-1}) \nonumber \\
&=& \lambda_{i} a_{i+1} +(1-\lambda_{i}) a_{i-1},\nonumber
 \end{eqnarray}
 where 
 $$\lambda_{i}= \frac{t_i-t_{i-1}}{t_{i+1}-t_{i-1}}=\frac{\Delta t_{i-1}}{\Delta t_{i} +\Delta t_{i-1}} \in (0,1).$$
 
 Let's deal now with \eqref{grv2}. By definition, $(t_i)_{i\geq 1}\in T_a$  if and only if
 $$
{ \frac{\Delta a_{i+1}}{\Delta t_{i+1}} \geq  \frac{\Delta a_{i}}{\Delta t_{i}}, \quad \text{for all}\; i\geq 1.}
 $$
 From this and the fact that $(a_i)_{i\geq 1}$ is increasing, we have
\begin{eqnarray}\label{in1}
{ \frac{\Delta a_{i+1}}{\Delta a_{i}} \leq  \frac{\Delta t_{i+1}}{\Delta t_{i}}, \quad \text{for all}\; i\geq 1.}
\end{eqnarray}
 Adding $ -1 $ to the both sides of \eqref{in1}, we obtain \eqref{grv2}.
\end{proof}

	One may see in Proposition~\ref{P1} that if the sequence $(t_i)_{i\geq 1}\in T_a$ is arithmetic, then \eqref{grv} and \eqref{grv2} reduce to \eqref{cv}. 
 
 The following result is a direct consequence of Proposition~\ref{P1}.
\begin{corollary}\label{Pro4}
	Let $(a_i)_{i\geq 1} \subset I$ be a real sequence, and let $\psi:  I \rightarrow \mathbb{R}$ be a non-decreasing convex function.
	Then $ T_a \subset T_{\psi(a)} $.
\end{corollary}
\begin{proof}
	Let $ (t_i)_{i\ge 1} \in T_a $. Then, the convexity of $ \psi $ and \eqref{grv}  yield, for all $i=\{2,\ldots,n-1\}$,
	\begin{eqnarray}\label{inq7}
		\psi(a_i) &\leq& \psi \left( \frac{\Delta t_{i}  }{\Delta t_{i} +\Delta t_{i-1}} a_{i-1}+\frac{ \Delta t_{i-1}  }{\Delta t_{i} +\Delta t_{i-1}}a_{i+1}\right) \nonumber\\
		& \leq& \frac{\Delta t_{i}  }{\Delta t_{i} +\Delta t_{i-1}} \psi (a_{i-1})+\frac{ \Delta t_{i-1}  }{\Delta t_{i} +\Delta t_{i-1}} \psi (a_{i+1}).
	\end{eqnarray}
	Thus $ (t_i)_{i\ge 1} \in T_{\psi(a)} $.
\end{proof}

The following result is a direct consequence of Theorem~\ref{rc-cf} and \cite[Theorem 1]{MV}.
\begin{proposition}\label{Prop2}
	Let $(a_i)_{i\geq 1}$ be a real sequence. Then, the following statements are equivalent
	\begin{itemize}
		\item [(1)] The sequence $(t_i)_{i\geq 1}\in T_a$;
		\item [(2)] The sequence $((a_i-a_s)/(t_i-t_s))_{i>s}$ is non-decreasing;
		\item [(3)] For any integers $l,m,n$ $(l<m<n)$, the following inequality holds
		\begin{eqnarray}\label{inq3}
		\left| \begin{array}{ccc}
		t_l & a_l &1 \\ t_m & a_m & 1\\ t_n&a_n&1
		\end{array}\right| =(t_n-t_m)a_l-(t_n-t_l)a_m+(t_m-t_l)a_n\geq 0.
		\end{eqnarray}
	\end{itemize}  
\end{proposition} 
\begin{proof} 
	By combining Theorem~\ref{rc-cf} and \cite[Theorem~1]{MV}, we obtain that (1) is equivalent to (3). In addition, by combining  Theorem~\ref{rc-cf} and \cite[Theorem~3]{MV} we obtain the equivalence between (1) and (2).
%
\end{proof}
It is known that if a sequence $(a_n)_{n\ge 1}$ is convex and bounded above, then $(a_n)_{n\ge 1}$ is non-increasing, see \cite{Bar,Z}. One may wonder if this property remains true for relative convex sequences.
\begin{proposition}\label{P11}
	Let $(a_n)_{n\ge 1}$ be a relative convex sequence and bounded above. Then, either $(a_n)_{n\ge 1}$ is non-increasing or every sequence $(t_n)_{n\ge 1} \in T_a$ is convergent.
\end{proposition}

\begin{proof}
	Let's assume the existence of $(t_n)_{n\ge 1} \in T_a$ that satisfies $t_n\to \infty$.
	We should prove that $\Delta a_n \leq 0$ for all integers $n\ge 1$. If this is not the case, there must exists an integer $m$ such that $\Delta a_m > 0$. By the definition of relative convexity, for any $k>m$, we have
	$$
	\frac{\Delta a_k}{\Delta t_k} \geq \frac{\Delta a_m}{\Delta t_m}>0.
	$$
	From this, one has
	$$
	a_n-a_m= \sum_{k=m}^{n-1}\Delta t_k \frac{\Delta a_k}{\Delta t_k}\geq \frac{\Delta a_m}{\Delta t_m} \sum_{k=m}^{n-1}\Delta t_k .
	$$
	Hence,
	$$
	a_n-a_m \geq (t_n-t_m)\frac{\Delta a_m}{\Delta t_m} \to \infty
	$$
	as $n\to \infty$, which contradicts the fact that $(a_n)_{n\ge 1}$ is bounded.
\end{proof}

Different from the convex sequences, the second case in Proposition~\ref{P11} may exist. For example, the sequence $ (\arctan n)_{n\ge 1} $ is relative convex since it's monotonic, and it is bounded above. As a consequence, every $(t_n)_{n\ge 1} \in T_a$ is convergent.

It is also known \cite{Bar,Z} that a bounded convex sequence $(a_n)_{n\ge 1}$ satisfies
$$
n\Delta a_n \to 0, \quad n\to \infty,
$$
and 
$$
\sum_{n=0}^{\infty}(n+1)\Delta^2 a_n <\infty.
$$

In the following result, similar properties are established for relative convex sequences.
\begin{proposition}\label{P2}
	Let $(a_n)_{n\ge 1}$ be a bounded relative convex sequence. If there exists a sequence $(t_n)_{n\ge 1} \in T_a$ satisfying 
	\begin{equation}\label{con}
		\liminf_{n\to+\infty} \Delta t_n >0.
	\end{equation}
	Then, 
	\begin{eqnarray}\label{note1}
		n\frac{\Delta a_n}{\Delta t_n}\to 0,\quad n\to \infty
	\end{eqnarray}
	and the series 
	$$
	\sum_{n\geq 1}n\Delta \left( \frac{\Delta a_n}{\Delta t_n}\right) 
	$$
	converges.
\end{proposition}

\begin{proof}
From the condition \eqref{con}, we deduce that $\lim\limits_{n\to\infty}t_n= \infty$. In addition, there exists $ \alpha>0 $ and $ N_1\in\N $ for which $ \Delta t_i \ge \alpha $ for every $ i\ge N_1 $.
From Proposition~\ref{P11}, we have $(a_n)_{n\ge 1}$ is non-increasing, and hence it is convergent, say that $a_n\to a$, $n\to \infty$. Therefore,
	$$
	a_1-a=\sum_{n=1}^{+\infty}-\frac{\Delta a_n}{\Delta t_n}\Delta t_n.
	$$
	Clearly, the sequence $(b_n)_{n\ge 1}$ defined by $b_n:=-\Delta a_n/\Delta t_n$ is non-negative non-increasing sequence. From this and the fact $\sum_{n=1}^{\infty} b_n \Delta t_n <\infty$, the sequence defined by 
	$$
	S_n=\sum_{i=1}^{n}b_i\Delta t_i
	$$
	is convergent. By the Cauchy Criterion, for all $\varepsilon>0$, there exists $N_2\ge N_1$ such that for all $m>n>N_2$, we have
	$$
	S_{m}-S_{n}=\left| \sum_{i=n+1}^{m}b_i\Delta t_i\right| < \varepsilon,
	$$	
	Now, by making use the fact that $(b_n)_{n\ge 1}$ is non-increasing we get
	$$
	(m-n)\alpha b_m \leq \left| \sum_{i=n+1}^{m}b_i\Delta t_i\right| < \varepsilon.
	$$
	Since $\alpha>0$  and by taking $m=2n$ one has
	$
	\lim\limits_{n\to \infty}2nb_{2n}= 0.
	$
	On the other hand
	$$
	(2n+1)b_{2n+1}=\frac{2n+1}{2n} 2nb_{2n+1}\leq 4 nb_{2n}\to 0, \quad n\to \infty.
	$$
	Therefore, $nb_n\to 0$ and \eqref{note1} follows immediately. 
	
	Observe that the series $\sum_{i=1}^{\infty}b_i<\infty$. Hence, by making use Abel's transformation, we find
	$$
	\sum_{i=1}^nb_i=\sum_{i=1}^{n}1 \cdot b_i=nb_n-\sum_{i=1}^{n-1}i\Delta b_i,
	$$
	and since $nb_n\to 0$, it follows that the series 
	$$
	\sum_{i=1}^{\infty}i\Delta\left( \frac{\Delta a_i}{\Delta t_i}\right) 
	$$
	converges.
\end{proof}

\section{Relative Convexity and Inequalities}
In this section, we focus on the use of relative convex sequences in the field of inequalities. Building upon some classical results for convex sequences such as Lupas and Hermite-Hadamard-Fej\'er inequalities, we prove analogous ones in the context of the relative convexity. Based on these new established results, one may drive even new versions of some inequalities for convex sequences as in Corollary~\ref{Cor2}  or Theorem~\ref{Cor3} below. 

\medskip

\subsection{Notation}\label{notation}
Before we state our results for this section, we need to define some notations.  For an integer $ n\ge 1 $, the vector $p=(p_1, \ldots, p_n) \in \R_+^n$ will always satisfy 
$$
P_n=\sum\limits_{i=1}^np_i>0.
$$
In addition, for $ x= (x_1, \ldots, x_n) \in \R^n$ and $ y=(y_1,\ldots,y_n) \in\R^n $ we use the notation
$$
M_{n,p}(x)=\frac{1}{P_n}\sum_{i=1}^np_ix_i \quad \text{and}\quad S_{n,p}(x,y)= M_{n,p}(xy) - M_{n,p}(x)M_{n,p}(y),
$$
where $ xy = (x_1y_1, \ldots, x_ny_n) $.

For any increasing sequence $t= (t_i)_{i\geq 1}$,  define the interval $ I_t $ by
\[
I_t =
\begin{cases}
	\left[t_1, t_n\right], & \text{if } t = \left(t_1, \ldots, t_n\right), \\
	\left[t_1, \lim\limits_{i\to \infty} t_i \right), & \text{if } t = \left(t_1, t_2, \ldots\right).
\end{cases}
\]
For any $q\in I_t$, we denote by $\lfloor q\rfloor_{t}$ the integer $i$ corresponding to the largest $t_i$ that is not grater than $q$.  In addition, the notation $\{q\}_t$ will stand for the quantity $q-t_{\lfloor q \rfloor_{t}}$. One can easily see that if $t_i=i$, then $\lfloor q \rfloor_{t}$ coincides with the ordinary floor function $\lfloor q \rfloor$, and  $\{q\}_t$ reduces to the fractional part function $\{q\}$.

The following simple example illustrates the new notation $ \lfloor\cdot\rfloor_{t} $.
\begin{example}
	Let $ t=(t_i)_{i\ge1} $ be an increasing sequence. 
	\begin{enumerate}
		\item If $ t_i = i-1 $, then $ \lfloor \pi \rfloor_t =4$ since $ t_4= 3 < \pi < 4=t_5 $. Meanwhile, we have $ \lfloor \pi \rfloor =3$. 
	
\item 	If $ t_i = \ln i $, then $ \lfloor 1/4\rfloor_t = 1 $ since $ t_1 =0< 1/4 < \ln 2 = t_2 $. Meanwhile, $ \lfloor 1/4\rfloor =0 $.
	\end{enumerate}
\end{example}


\subsection{Lupas inequality}
For two integrable functions $f, g:[a, b] \rightarrow \mathbb{R}$, consider the Chebyshev functional
$$
C(f, g):=\frac{1}{b-a} \int_a^b f(x) g(x) d x-\frac{1}{(b-a)^2} \int_a^b f(x) d x \cdot \int_a^b g(x) d x .
$$
In 1972, Lupas \cite{Lu} showed that if $f, g$ are convex functions on the interval $[a, b]$, then
\begin{eqnarray}\label{Lup}
	C(f, g) \geq \frac{12}{(b-a)^4} \int_a^b\left(x-\frac{a+b}{2}\right) f(x) d x  \int_a^b\left(x-\frac{a+b}{2}\right) g(x) dx
\end{eqnarray} 
with equality when at least one of the functions $f, g$ is a linear function on $[a, b]$. A discrete version of \eqref{Lup} was given later by Pe\u{c}ari\'{c} \cite{P}. Indeed, he showed that the inequality
\begin{equation}\label{peca}
\sum_{i=1}^n a_i b_i-\frac{1}{n} \sum_{i=1}^n a_i \sum_{i=1}^n b_i \geq \frac{12}{n\left(n^2-1\right)} \sum_{i=1}^n\left( i-\frac{n+1 }{2}\right) a_i \cdot \sum_{i=1}^n\left( i-\frac{n+1 }{2}\right) b_i
\end{equation}
holds for convex sequences $(a_i)_{i=1}^n$ and $(b_i)_{i=1}^n$, with equality when at least one of $(a_i)_{i=1}^n$ and $(b_i)_{i=1}^n$ is an arithmetic sequence.

\medskip
Motivated by the above inequalities, we provide the following version of Lupas inequality for relative convex sequences.
\begin{theorem}\label{Th5}
	Let $a=(a_1,\ldots,a_n)$ and $b=(b_1,\ldots,b_n)$ be relative convex sequences. If $t=(t_1,\ldots,t_n) \in T_a \cap T_b$, then
\begin{eqnarray}\label{Lupg}
	S_{n,p}(a,b)\geq \frac{S_{n,p}(a,t)\ S_{n,p}(b,t)}{S_{n,p}(t,t)}.
\end{eqnarray}
\end{theorem}

\begin{proof}
	
	From Proposition~\ref{Prop2}, and for arbitrary indices $l,m,k$, we have
	\begin{equation*}
		\left| \begin{array}{ccc}
			t_l & a_l &1 \\ t_m & a_m & 1\\ t_k&a_k&1
		\end{array}\right| \cdot \left| \begin{array}{ccc}
		t_l & b_l &1 \\ t_m & b_m & 1\\ t_k&b_k&1
	\end{array}\right|  \ge 0,
	\end{equation*}
	which is equivalent to
	$$
	\Bigg( (t_k-t_m)a_l-(t_k-t_l)a_m+(t_m-t_l)a_k \Bigg)       \Bigg( (t_k-t_m)b_l-(t_k-t_l)b_m+(t_m-t_l)b_k \Bigg)     \geq 0.
	$$
	Now, following the proof provided in \cite{Lu}, that is, multiplying the above inequality by $ p_l p_m p_k $ and then summing with respect to each index $ l,m,k $ from $ 1 $ to $ n $, we get \eqref{Lupg}.
\end{proof}

The following corollary is a direct consequence of Theorem~\ref{Th5}.
\begin{corollary}
	Under the hypotheses of Theorem~\ref{Th5}, and for $ p = (1, \ldots,1) $, we obtain 
	\begin{eqnarray}\label{Lup2}
		\sum_{i=1}^n a_i b_i-\frac{1}{n} \sum_{i=1}^n a_i \sum_{i=1}^n b_i \geq K_n(t)\sum_{i=1}^n\left(t_i-\frac{\sum_{i=1}^n t_i}{n} \right) a_i \cdot \sum_{i=1}^n\left(t_i-\frac{\sum_{i=1}^n t_i}{n} \right) b_i,
	\end{eqnarray}
	
	where 
	$$
	K_n(t)=\frac{1}{\sum_{i=1}^nt_i^2-\frac{1}{n}(\sum_{i=1}^nt_i)^2}.
	$$
	Moreover, if $(t_i)_{i=1}^n$ is an arithmetic, then the inequality \eqref{Lup2} reduces to \eqref{peca}.
\end{corollary}


\subsection{Hermite-Hadamard-Fej\'{e}r type inequalities }
The classical Hermite-Hadamard-Fej\'er inequality gives an estimate of the weighted mean value of a convex function  $f:[a, b] \rightarrow \mathbb{R}$ with respect to the weight function $p:[a, b] \rightarrow [0, \infty)$. In fact, if $ p(x) $ is an integrable function and symmetric about $({a+b})/{2}$, then
\begin{equation}\label{HH}
	f\left(\frac{a+b}{2}\right) \int_a^b p(x) d x \leq \int_a^b p(x) f(x) d x \leq \frac{f(a)+f(b)}{2} \int_a^b p(x) d x.
\end{equation}
Fore more details, see \cite[Chapter 5]{PPT}. 

\medskip 
The discrete counterpart of \eqref{HH} for convex sequences is given in \cite[Theorem 1.4]{LB}. Precisely, the following inequality is established
\begin{equation}\label{d}
	\frac{a_{m+1}+a_{m}}{2} \leq \frac{1}{P_n}\sum_{i=1}^n p_ia_i\leq \frac{a_1+a_n}{2},
\end{equation}
where $ m=\left\lfloor ({n+1})/{2} \right\rfloor $, $(a_i)_{i=1}^n$ is a convex sequence and $(p_i)_{i=1}^n$  is symmetric about $(n+1)/2$, that is, $p_i=p_{n+1-i}$ , for all $i=1,\ldots,n$.
Notice that in \cite{LB} the lower bound in \eqref{d} is given as $({a_m+a_{n+1-m}})/{2}$. However, one can easily check that these quantities are equal.

\medskip

Later on, Niezgoda \cite[Theorem 3.1]{N} employed some matrix methods based on column stochastic and doubly stochastic matrices to obtain the following general version to the right hand side of \eqref{d}, without the symmetry condition on $(p_i)_{i=1}^n$. In fact, for a non-decreasing convex function $\psi: I \rightarrow \mathbb{R}$ defined on an interval $I\subset \mathbb{R}$, and for a convex sequence $(a_i)_{i=1}^n\subset I$,  the following inequality holds 
\begin{eqnarray}\label{c5}
	\sum_{i=1}^n p_i\psi( a_i)\leq\left( \sum_{i=1}^n\frac{n-i}{n-1}p_i\right) \psi(a_1)+\left( \sum_{i=1}^n\frac{i-1}{n-1}p_i\right) \psi(a_n).
\end{eqnarray} 

%


\medskip
Motivated by the above results, we give now an analogue of Hermite-Hadamard-Fej\'{e}r inequality for relative convex sequences.

\begin{theorem}\label{Th1}
	Suppose that $\psi: I \rightarrow \mathbb{R}$ be a non-decreasing convex function defined on an interval $I\subset \mathbb{R}$. Let $a=(a_1,\ldots,a_n)\subset I^n$ be a relative convex sequence.
	Then, for every $t=(t_1, \ldots,t_n) \in T_a$, we have
\begin{eqnarray}\label{e6}
\gamma_t \psi (a_{m+1})+(1-\gamma_t)\psi (a_{m})	  \leq \frac{1}{P_n}	\sum_{i=1}^n p_i \psi (a_i)
\leq \lambda _t \psi (a_1)+(1-\lambda_t)  \psi(a_n),
\end{eqnarray} 
where
$$
m=\lfloor M_{n,p}(t)\rfloor_{t}, \quad\gamma_t=\frac{M_{n,p}(t)-t_m}{t_{m+1}-t_m}\quad \text{and}\quad \lambda_t=\frac{t_n-M_{n,p}(t)}{t_n-t_1}.
$$
\end{theorem}
\begin{proof}[Proof of Theorem~\ref{Th1}]
Since $(t_i)_{i=1}^n \in T_a$ and by using the convexity of $a_t^*$, the inequality
\begin{eqnarray}
	a_i= a_t^*(t_i)&=&a_t^*\left(\frac{t_n-t_i}{t_n-t_1} t_1 +\frac{t_i-t_1}{t_n-t_1} t_n \right) \nonumber \\
	&\leq & \frac{t_n-t_i}{t_n-t_1} a_t^*(t_1)+ \frac{t_i-t_1}{t_n-t_1} a_t^*(t_n)\nonumber \\
	&=& \frac{t_n-t_i}{t_n-t_1} a_1+ \frac{t_i-t_1}{t_n-t_1} a_n\nonumber
\end{eqnarray}
holds for any $i=1,2,\ldots,n$. This together with Corollary~\ref{Pro4} imply that
\begin{eqnarray}\label{k1}
	\psi(a_i) \leq \frac{t_n-t_i}{t_n-t_1} \psi (a_1)+ \frac{t_i-t_1}{t_n-t_1} \psi (a_n)
\end{eqnarray} 
holds for any $i=1,2,\ldots,n$. Hence, multiplying both sides of \eqref{k1} by $\frac{p_i}{P_n}$ and summing with respect to $i$ from $1$ to $n$, yield the right hand side of \eqref{e6}.

\medskip
On the other hand, from the convexity of $a_t^*$  we have
\begin{eqnarray}
\frac{1}{P_n}	\sum_{i=1}^n p_i a_i= \sum_{i=1}^n \frac{p_i}{P_n} a_t^*(t_i) &\geq& a_t^*(M_{n,p}(t))) = a_{m}+\left(M_{n,p}(t)-t_m\right)\frac{\Delta a_{m}}{\Delta t_{m}},\nonumber
\end{eqnarray} 
where $ m=\lfloor M_{n,p}(t)\rfloor_{t} $, that is, $ t_m \le M_{n,p}(t)< t_{m+1} $. From this, we deduce 
$$
\frac{1}{P_n}	\sum_{i=1}^n p_i a_i \geq \gamma_ta_{m+1}+(1-\gamma_t)a_m.
$$
Consequently, by making use of Corollary~\ref{Pro4} we obtain the left hand side of \eqref{e6}. 
\end{proof}
If $t=(t_1, \ldots,t_n)$ is an arithmetic sequence, then the right inequality in \eqref{e6} simplifies to \eqref{c5}. Furthermore, we can derive the following corollary from Theorem~\ref{Th1}, which improves \cite[Theorem 1.4]{LB} and \cite[Theorem 3.4]{N}.
\begin{corollary}\label{Cor2}
	Suppose that $\psi: I \rightarrow \mathbb{R}$ be a non-decreasing convex function defined on an interval $I\subset \mathbb{R}$. Then, for every convex sequence $a=(a_1,\ldots,a_n)\subset I^n$, we have
		\begin{eqnarray}\label{NHH}
\Phi(m,m+1)\leq 	\sum_{i=1}^n p_i\psi( a_i)
	\leq \Phi(1,n),
	\end{eqnarray} 
	where 
	$$
	 m=\left\lfloor \frac{1}{P_n}\sum\limits_{i=1}^np_ii\right\rfloor \quad \text{and}\quad \Phi(u,v)=\left( \sum_{i=1}^n\frac{i-u}{v-u}p_i\right) \psi(a_v)+ \left( \sum_{i=1}^n\frac{v-i}{v-u}p_i\right) \psi(a_u).
	$$
\end{corollary}
	In Corollary~\ref{Cor2}, if $p=(p_1,\ldots,p_n)$ is symmetric about $(n+1)/2$, then $m=\lfloor(n+1)/2\rfloor$ and \eqref{NHH} reduces to
$$
\frac{\psi(a_{m+1})+\psi(a_{m})}{2} \leq \frac{1}{P_n}\sum_{i=1}^n p_i\psi(a_i)\leq \frac{\psi(a_1)+\psi(a_n)}{2}.
$$

\subsection{Inequalities related to the majorization theory.}
For $ x=\left(x_1, x_2, \ldots, x_n\right) \in \mathbb{R}^n $, and  $ y=\left(y_1, y_2, \ldots, y_n\right) \in \mathbb{R}^n  $, let $ x_{[1]} \geq x_{[2]} \geq\cdots \geq x_{[n]} $ and $ y_{[1]} \geq y_{[2]} \geq \cdots \geq y_{[n]} $ denote the components of $ x $ and $ y $ in decreasing order,
respectively. We say that $ y $ majorizes $ x $ (or $ x $ is majorized by $ y $), and we write $ x \prec y $, if 
\begin{equation*}
	\sum_{i=1}^k x_{[i]} \leq \sum_{i=1}^k y_{[i]}, \,\  k=1,2, \ldots, n-1, \quad \text { and } \quad \sum_{i=1}^n x_i=\sum_{i=1}^n y_i.
\end{equation*}
Here, we are interested, in particular, in the following result due to Schur, Hardy-Littlewood-Polya.


\begin{theorem}\cite[Propositions C.1. and C.1.c.]{MOA}\label{lSch}
	Let $f$ be a real function continuous  on an interval $ I $. Then $f$ is convex on $I$ if and only if 
	$$
	\sum_{i=1}^n f(x_i) \le \sum_{i=1}^n f(y_i)
	$$
holds for all $ n\ge 2 $ and all  $ (x_1, \ldots, x_n), (y_1, \ldots, y_n) \in I^n $  satisfying $ (x_1, \ldots, x_n)\prec (y_1, \ldots, y_n)$.
\end{theorem}

\medskip
A version of Theorem~\ref{lSch} for convex sequences is given in the following result.

\begin{theorem}[{\cite[Theorem~2]{WD}, \cite[Theorem 1.1.9]{HNS}}] \label{seq-Schur}
Let $\left(a_i\right)_{i\ge1}$ be a real sequence. Then $\left(a_i\right)_{i\ge1}$ is convex if and only if
	$$
	\sum_{i=1}^n a_{p_i} \leq 	\sum_{i=1}^n a_{q_i}
	$$
	holds for all $ n\ge 2 $ and all $ (p_1, \ldots, p_n) , (q_1, \ldots, q_n) \in \N ^n$ satisfying $\left(p_1, \ldots, p_n\right) \prec\left(q_1, \ldots, q_n\right)$.
\end{theorem}

In the following result, we generalize Theorem~\ref{seq-Schur} in the sense that the components of the vectors $ (p_1, \ldots, p_n) $ and $ (q_1, \ldots, q_n) $ can take non-negative real values.

\begin{theorem}\label{Cor3}
	Let  $\left(a_i\right)_{i\ge1}$ be a real sequence.  Then $(a_i)_{i\geq 1}$ is convex if and only if
	\begin{eqnarray}\label{Maj2}
		\sum_{i=1}^n \left( a_{\lfloor p_i \rfloor }-a_{\lfloor q_i \rfloor } \right) \leq 	 \sum_{i=1}^n \Big( \{q_i\} \Delta a_{\lfloor q_i\rfloor} -\{p_i\}\Delta a_{\lfloor p_i\rfloor}\Big) 
	\end{eqnarray}
	holds for all $ n\ge 2 $ and all $ (p_1, \ldots, p_n) , (q_1, \ldots, q_n) \in [1,+\infty)^n$ satisfying  $ \left(p_1, \ldots, p_n\right) \prec\left(q_1,\ldots,q_n\right) $.
\end{theorem}

Instead of proving Theorem~\ref{Cor3}, we will prove the following result, which gives an equivalence for relative convex sequences by means of the majorization. In fact, Theorem~\ref{Cor3} can be derived directly from Theorem~\ref{Th_Maj} below by noticing that  $ (a_i)_{i\ge 1} $ is convex if and only if $ (i)_{i\ge 1} \in T_a $.

\begin{theorem}\label{Th_Maj}
	Let $ (a_i)_{i\ge1} $ be a real sequence. Then $ (t_i)_{i\ge 1} \in~ T_a$ if and only if the following inequality
	\begin{equation}\label{Maj}
		\sum_{i=1}^n \left( a_{\lfloor p_i \rfloor_t }-a_{\lfloor q_i \rfloor_t } \right) \leq 	 \sum_{i=1}^n \left( \{q_i\}_t\frac{\Delta a_{\lfloor q_i\rfloor_{t}}}{\Delta t_{\lfloor q_i\rfloor_{t}}}-\{p_i\}_t\frac{\Delta a_{\lfloor p_i\rfloor_{t}}}{\Delta t_{\lfloor p_i\rfloor_{t}}}\right) 
	\end{equation}
	holds for any $ (p_1, \ldots, p_n) , (q_1, \ldots, q_n) \in I_t^n$ satisfying  $ \left(p_1, \ldots, p_n\right) \prec\left(q_1,\ldots,q_n\right). $
\end{theorem}

\begin{proof}
%
	From Theorem~\ref{rc-cf} and its proof,  we know that $ (t_i)_{i\ge 1} \in T_a$ is equivalent to say that the polygonal function $a_t^*$ defined in \eqref{defp} is convex on $I_t$.  Theorem~\ref{lSch} shows that 
	\begin{eqnarray}\label{ine}
	\sum_{i=1}^n a_t^*(p_i) \leq \sum_{i=1}^na_t^*(q_i)
\end{eqnarray}
	holds for all $ n\ge 2 $ and for all $ (p_1, \ldots, p_n) , (q_1, \ldots, q_n) \in I_t^n$  satisfying  $ \left(p_1, \ldots, p_n\right) \prec\left(q_1,\ldots,q_n\right) $. By adopting the notations from Subsection~\ref{notation} we can write
	$$
		a_t^*(q)=a_{\lfloor q\rfloor_{t}}+\{q\}_t\frac{\Delta a_{\lfloor q\rfloor_{t}}}{\Delta t_{\lfloor q\rfloor_{t}}}, \quad  q\in I_t.
	$$
Plugging this into \eqref{ine}, we obtain \eqref{Maj}.
\end{proof}





\begin{thebibliography}{99}


    \bibitem{Bar}
Bary, N. K.,
\emph{A Treatise on Trigonometric Series. Vols. I, II.}. 
Pergamon Press Book, Macmillan, New York, 1964

	\bibitem{cargo}
Cargo,~G.~T., 
\emph{Comparable means and generalized convexity}.
J. Math. Anal. App. \textbf{12} (1965), 387--392. 


	

	\bibitem{HLP}
Hardy, G. H., Littlewood, J. E. and Polya, G.,
\emph{Inequalities}, second ed.
Cambridge University Press, 1952.


	\bibitem{LB}
Latreuch, Z., Belaidi, B.,
\emph{New inequalities for convex sequences with applications}. 
Int. J.	Open Problems Comput. Math. \textbf{5} (2012), 15--27.
	
		\bibitem{Lu}
	Lupas, A.,
	\emph{An integral inequality for convex functions}. 
	Publ. Fac. Electrotech. Univ. Belgrade, Ser. Math. Phys. \textbf{17-19} (1972), 381--409.
	
	
	\bibitem{MOA}
	Marshall, A. W., Olkin I. and Arnold, B. C.,
	\emph{Inequalities: The Theory of Majorization and its Applications}, 
	second ed. 
	Springer, New York, 2011.
	


	

	

	

	
	\bibitem{MPF}
	Mitrinovi\'{c}, D. S., 	Pe\u{c}ari\'{c}, J. E. and Fink, A. M.,
	\emph{Classical and New Inequalities in Analysis}.
	Kluwer Academic Publishers,  1993.
	
	
	
		\bibitem{MV}
	Mitrinovi\'{c}, D. S. (in cooperation with Vasić., P. M.),
	\emph{Analytic Inequalities}. 
	Springer-Verlag, Berlin, 1970.
	
	\bibitem{NP}
	Niculescu, C.~P. and Persson, L.~E.,
	\emph{Convex Functions and Their Applications, A Contemporary Approach}. 
	Springer-Verlag, New York, 2006.

	\bibitem{N}
Niezgoda, M.,
\emph{Sherman, Hermite-Hadamard and Fejér like inequalities for convex sequences and non-decreasing convex functions}. 
Filomat \textbf{31} (2017), 2321--2335. 


	\bibitem{P}
Pe\u{c}ari\'{c}, J. E.,
\emph{On some inequalities for convex sequences}. 
Publ. Inst. Math., Nouv. Sér. \textbf{ 33}(47) (1983), 173--178 .

\bibitem{PPT}
Pe\u{c}ari\'{c}, J. E., Proschan, F. and Tong, Y. L.,
\emph{Convex Functions, Partial Orderings, and Statistical Applications}. 
 Academic Press, Inc., Boston, 1992.


	
		\bibitem{RS}
	Rustogi, K. and Strusevich, V. A.,
	\emph{Convex and V-shaped sequences of sums of functions that depend on ceiling functions}. 
	J. Integer Sequences \textbf{14} (2011), Article 11.1.5.


\bibitem{HNS}
Shi, H. N., 
\emph{Schur-Convex Functions and Inequalities}. \emph{Volume 2: Applications in inequalities.} 
Berlin De Gruyter, Harbin Institute of Technology Press, 2019.


	
	\bibitem{WD}
	Wu, S. and Debnath, L.,
	\emph{Inequalities for convex sequences and their applications}. 
	Comput. Math. Appl. \textbf{54} (2007), 525--534. 
	

	

	

	

	

	

	

	

	
	\bibitem{Z}
	Zygmund, A.,
	\emph{Trigonometric series. Vol. I, II}. 
3rd edn. Cambridge University Press, Cambridge (2002).
\end{thebibliography}
\end{document}